\documentclass[11pt,reqno]{amsart}

\usepackage[top=30truemm,bottom=30truemm,left=30truemm,right=30truemm]{geometry}

\usepackage[driverfallback=dvipdfm]{hyperref}
\usepackage{amscd}
\usepackage{amsmath, amssymb}
\usepackage{mathrsfs}
\usepackage{multicol}
\usepackage{ifthen}
\usepackage{threeparttable}
\usepackage{caption}
\usepackage{bm}
\usepackage[all]{xy}
\usepackage{amsthm}
\usepackage[dvipsnames]{xcolor}
\usepackage{url}
\usepackage{paralist} 
\usepackage{enumerate}
\makeindex
\usepackage{latexsym}
\usepackage{arydshln}

\usepackage[backend=bibtex, style=numeric, sorting=nty, giveninits=true, backref=false, backrefstyle=none, doi=false]{biblatex} 
\bibliography{mizuno_paper3.2}
\AtBeginBibliography{\footnotesize}\DeclareFieldFormat{biblabeldate}{\mkbibparens{\mkbibbold{#1}}}\DeclareFieldFormat[article,periodical]{volume}{\mkbibbold{#1}}\DeclareFieldFormat[article,periodical]{number}{\mkbibbold{#1}}\renewbibmacro{in:}{}

\hypersetup{setpagesize=true, bookmarksnumbered=false, bookmarksopen=true, colorlinks=true,linkcolor=RubineRed, citecolor=LimeGreen,}

\newtheorem{thm}{Theorem}[section]

\newtheorem{cor}[thm]{Corollary}

\newtheorem{slem}[thm]{Sub-Lemma}
\newtheorem{prop}[thm]{Proposition}

\newtheorem{lemm}[thm]{Lemma}

\theoremstyle{definition}
\newtheorem{dfn}[thm]{Definition}
\newtheorem*{dfn*}{Definition}
\newtheorem{nota}[thm]{Notation}
\newtheorem{eg}[thm]{Example}
\newtheorem{notaterm}[thm]{Notation and Terminology}

\theoremstyle{remark}
\newtheorem{rmk}[thm]{Remark}

\newcommand{\la}{\langle}
\newcommand{\ra}{\rangle}
\newcommand{\op}{\mathrm}

\numberwithin{equation}{subsection}

\usepackage{multicol}
\usepackage{listings}
\usepackage{lineno}

\begin{document}

\title[Some examples of noncommutative Calabi-Yau schemes]
{Some examples of noncommutative projective Calabi-Yau schemes} 
\author{Yuki Mizuno}
\email{m7d5932a72xxgxo@fuji.waseda.jp}
\date{}
\address{Department~of~Mathematics, School~of~Science~and~Engineering, Waseda~University, Ohkubo~3-4-1, Shinjuku, Tokyo~169-8555, Japan}

\keywords{Noncommutative algebraic geometry, Calabi-Yau varieties}
\subjclass[2020]{14A22, 14J32}

\maketitle

\begin{abstract}
In this article, we construct some examples of noncommutative projective Calabi-Yau schemes by using noncommutative Segre products and quantum weighted hypersurfaces.
We also compare our constructions with commutative Calabi-Yau varieties and examples constructed in \cite{kanazawa2015}.
In particular, we show that some of our constructions are essentially new examples of noncommutative projective Calabi-Yau schemes.
\end{abstract}

\section{Introduction}
Calabi-Yau varieties are rich objects and play an important role in mathematics and physics.
In noncommutative geometry, (skew) Calabi-Yau algebras are often treated as noncommutative analogues of Calabi-Yau varieties.
Calabi-Yau algebras have a deep relationship with quiver algebras (\cite{ginzburg2006calabi}, \cite{vandenbergh2010calabi}).
For example, many known Calabi-Yau algebras are constructed by using quiver algebras.
They are also used to characterize Artin-Schelter regular algebras (\cite{reyes2022graded}, \cite{reyes2014}).
In particular, a connected graded algebra $A$ over a field $k$ is Artin-Schelter regular if and only if $A$ is skew Calabi-Yau.

On the other hand, a triangulated subcategory of the derived category of a cubic fourfold in $\mathbb{P}^5$,  which is obtained by some semiorthogonal decompositions, has the 2-shift functor $[2]$ as the Serre functor.
Moreover, the structure of Hochschild (co)homology is the same as that of a projective K3 surface (\cite{kuznetsov2019calabi}).
However, some such categories are not obtained as the derived categories of coherent sheaves of projective K3 surfaces and called noncommutative K3 surfaces.

Artin and Zhang constructed a framework of noncommutative projective schemes in \cite{artin1994}, which are defined from noncommutative graded algebras.
In this framework, we can think of Artin-Schelter regular algebras as noncommutative analogues of projective spaces, which are called quantum projective spaces.
Our objective is to produce examples of noncommutative projective Calabi-Yau schemes that are not obtained from commutative Calabi-Yau varieties.
In the future, it would be an interesting question to compare the derived category of a noncommutative projective Calabi-Yau scheme created in the framework of Artin-Zhang's noncommutative projective schemes with 
 a noncommutative K3 surface obtained as a triangulated subcategory of the derived category of a cubic fourfold.

As the definition of noncommutative projective Calabi-Yau schemes, we adopt the definition introduced by Kanazawa (\cite{kanazawa2015}).
His definition is a direct generalization of the definition of commutative Calabi-Yau varieties to noncommutative projective schemes.
He also constructed the first examples of noncommutative projective Calabi-Yau schemes that are not isomorphic to commutative Calabi-Yau varieties as hypersurfaces of quantum projective spaces.
Recently, some examples constructed by Kanazawa play an important role in noncommutative Donaldson-Thomas theory (\cite{liu2019donaldson}, \cite{liu2020donaldson}).

In this paper, we construct new examples of noncommutative projective Calabi-Yau schemes by using noncommutative Segre products and weighted hypersurfaces.
There are many known examples of Calabi-Yau varieties in algebraic geometry.
Some of them are complete intersections in Segre embeddings of products of projective spaces.
Moreover, Reid gave a list of Calabi-Yau surfaces, which are hypersurfaces in weighted projective spaces (\cite{iano2000}, \cite{reid1980}).
Motivated by these two facts, we construct noncommutative analogues of the two types of examples of Calabi-Yau varieties (Theorem \ref{thm1}, Theorem \ref{thm2}) in Section \ref{sec2}.

In order to prove that a noncommutative projective scheme is Calabi-Yau, we use the methods of Kanazawa. 
However, they are not sufficient because the algebras we treat are more complicated than the ones he considered.
In order to construct noncommutative projective Calabi-Yau schemes as noncommutative analogues of complete intersections in Segre products, we perform a more detailed analysis of noncommutative projective schemes defined by $\mathbb{Z}^2$-graded algebras, which were studied by Van Rompay (\cite{rompay1996483}).
A different approach to noncommutative Segre products is also studied in \cite{he2021twisted}. 
In order to construct noncommutative projective Calabi-Yau schemes as noncommutative analogues of weighted hypersurfaces, we consider quotients of weighted quantum polynomial rings.
In commutative algebraic geometry, the projective spectrum $\op{Proj}(k[x_0,\cdots,x_n])$ of a weighted polynomial ring is not necessarily isomorphic to $\op{qgr}(k[x_0,\cdots,x_n])$, where $\op{qgr}(k[x_0,\cdots,x_n])$ is the quotient category associated to $k[x_0,\cdots,x_n]$ constructed in \cite{artin1994}.
However, $\op{qgr}(k[x_0,\cdots,x_n])$ is thought of as a nonsingular model of $\op{Proj}(k[x_0,\cdots,x_n])$ (see \cite[Example 4.9]{smith2004maps}).
We use this idea to construct new noncommutative projective Calabi-Yau schemes.
In addition, it should be noted that local structures of noncommutative projective schemes of quotients of weighted quantum polynomial rings are somewhat complicated.
An analysis of the local structures was performed by Smith (\cite{smith2004maps}).
We show that the local structure obtained in \cite{smith2004maps} is described by the notion of quasi-Veronese algebras introduced by Mori (\cite{mori2013b}).

In Section \ref{ptmods}, we compare our constructions from weighted hypersurfaces in Section \ref{sec2} with commutative Calabi-Yau varieties and the first examples constructed in \cite{kanazawa2015}, focusing on noncommutative projective Calabi-Yau schemes of dimensions 2.
We show that some of our constructions in Section \ref{sec2} are not isomorphic to any of the commutative Calabi-Yau varieties and the first examples constructed in \cite{kanazawa2015} (Proposition \ref{110549_8Sep23}).
When we consider moduli spaces of point modules of noncommutative projective schemes obtained from weighted hypersurfaces in Section \ref{sec2}, there is a problem, which is in general weighted quantum polynomial rings are not generated in degree $1$.
So, the notion of point modules is not necessarily useful in this case.
In this paper, we use theories of closed points studied in \cite{mori2015regular}, \cite{smith-noncommutative} and \cite{smith2002subspaces}, etc.
A different approach to closed points of weighted quantum polynomial rings is studied in \cite{stephenson200070}.
The notion of point modules defined in \cite{stephenson200070} corresponds to those of ordinary and thin points in \cite{mori2015regular}.
To show that some of our constructions are not isomorphic to the examples obtained in \cite{kanazawa2015}, we use Morita theory of noncommutative schemes, which is established in \cite{burban2022morita} (see also \cite[Section 6]{artin1994}).
In the theory, we need to calculate the centers of noncommutative rings.
By using these calculations, we can do a detailed analysis and some classifications of noncommutative projective Calabi-Yau surfaces.


\section{Preliminaries}
\label{sec1}
\begin{notaterm}
 In this article, $k$ means an algebraically closed field of characteristic 0.
We suppose $\mathbb{N}$ contains $0$.
Let $A$ be a $k$-algebra, $M$ be an $A$-bimodule and $\psi, \phi$ be algebra automorphisms of $A$.
Then, we denote the associated $A$-bimodule by $^{\psi}M^{\phi}$, i.e. ${}^\psi M^{\phi} = M$ as $k$-modules and the new bimodule structure is given by $a*m*b := \psi(a)m\phi(b)$ for all $a,b \in A$ and all $m \in M$.
Let $\mathcal{C}$ be a $k$-linear abelian category.
We denote the global dimension of $\mathcal{C}$ by $\op{gl.dim}(\mathcal{C})$.
An $\mathbb{N}$-graded $k$-algebra $A$ is connected if $A_0=k$. 

For any $\mathbb{N}$-graded $k$-algebra $A = \bigoplus_{i=0}^{\infty}A_i$, we denote the category of graded right $A$-modules (resp. finitely generated graded right $A$-modules) by $\op{Gr}(A)$ (resp. $\op{gr}(A)$).
 Let $M \in \op{Gr}(A)$ and $A^\circ$ be the opposite algebra of $A$. 
 We define the Matlis dual $M^* \in \op{Gr}(A^\circ)$ by $M^*_i := \op{Hom}_{k}(M_{-i},k)$ and the shift $M(n) \in \op{Gr}(A)$ by $M(n)_i :=  M_{i+n}  \ (i,n \in \mathbb{Z})$.
For $M,N \in \op{Gr}(A)$, we write ${\op{Hom}}_A(M,N) := \bigoplus_{n \in \mathbb{Z}} \op{Hom}_{\op{Gr(A)}}(M,N(n)) \in \op{Gr}(A)$.
For $M \in \op{Gr}(A)$ and a homogeneous element $m \in M$, we denote the degree of $m$ by $\op{deg}(m)$.
We define the truncation $M_{\geq n} := \bigoplus_{i \geq n} M_i \in \op{Gr}(A) \ (n \in \mathbb{Z})$. 
An element $m \in M$ is called torsion if $mA_{\geq n}=0$ for $n \gg 0$.
We say $M$ is a torsion module if any element of $M$ is torsion. 
We denote the subcategory of torsion modules in $\op{Gr}(A)$ (resp. $\op{gr}(A)$) by $\op{Tor}(A)$ (resp. $\op{tor}(A)$). 
\end{notaterm}

\begin{dfn}[{\cite[Section 2]{artin1994}}]
 Let $A$ be a right noetherian $\mathbb{N}$-graded $k$-algebra.
We define the quotient categories $\op{QGr(A)} :=\op{Gr}(A)/\op{Tor}(A)$ and $\op{qgr(A)}:=\op{gr}(A)/\op{tor}(A)$. 
We denote the projection functor by $\pi$ and  its right adjoint functor by $\omega.$
The general (resp. noetherian) projective scheme of $A$ is defined as $\op{Proj}(A) := (\op{QGr}{(A)}, \pi(A))$ (resp. $\op{proj}(A) := (\op{qgr}{(A)}, \pi(A))$).
\label{no.1}
\end{dfn}

\begin{dfn}[{\cite[Section 2]{artin1994}, \cite[Chapter 3]{smith-noncommutative}}]
A quasi-scheme over $k$ is a pair $(\mathcal{C}, \mathcal{O})$ where $\mathcal{C}$ is a $k$-linear abelian category and $\mathcal{O}$ is an object in $\mathcal{C}$.
A morphism from a quasi-scheme $(\mathcal{C}, \mathcal{O})$ to another quasi-scheme $(\mathcal{C}', \mathcal{O}')$ is a pair $(F,\varphi)$ consisting of a $k$-linear right exact functor $F:\mathcal{C} \rightarrow \mathcal{C}'$ and an isomorphism $\varphi: F(\mathcal{O}) \overset{\simeq}{\rightarrow} \mathcal{O}'$. We call $(F,\varphi)$ is an isomorphism if $F$ is an equivalence.

When $A$ is as in Definition \ref{no.1}, we think of $\op{proj(A)} = (\op{qgr}(A), \pi(A))$ as a quasi-scheme.
For any (commutative) noetherian scheme $X$, $(\op{Coh}(X), \mathcal{O}_X)$ is also a quasi-scheme.
From this observation, we regard $X$ as a quasi-scheme.
\end{dfn}

\begin{dfn}[{\cite[Section 4]{vandenbergh1997}, \cite[Section4]{yekutieli1992}}]
 Let $A,B$ be $\mathbb{N}$-graded $k$-algebras and $m_{A}$ be $A_{\geq 1}$.
We define the torsion functor $\Gamma_{m_A}: \op{Gr}(A \otimes_k B^\circ) \rightarrow \op{Gr}(A \otimes_k B^\circ)$ by $\Gamma_{m_A}(M) := \{m \in M \mid mA_{\geq n}  = 0 \text{ for some }n \in \mathbb{N} \}$.
We write $H^i_{m_A}:=\op{R}^i\Gamma_{m_A}$.
\end{dfn}

\begin{dfn}[{\cite[Definition 6.1, 6.2]{vandenbergh1997}, \cite[Definition 3.3, 4.1]{yekutieli1992} }]
 Let $A$ be a right and left noetherian connected $\mathbb{N}$-graded $k$-algebra and $A^e$ be the enveloping algebra of $A$.
Let $R$ be an object of $\op{D^b}(\op{Gr}(A^e))$.
Then, $R$ is called a dualizing complex of $A$ if 
\begin{inparaenum}[(1)]
 \item $R$ has finite injective dimension over $A$ and $A^\circ$,
\item The cohomologies of $R$ are finitely generated as both $A$ and $A^\circ$-modules,
\item The natural morphisms $A \rightarrow \op{RHom}_A(R,R)$ and $A \rightarrow \op{RHom}_{A^\circ}(R,R)$ are isomorphisms in $\op{D^b}(\op{Gr}(A^e))$.
\end{inparaenum}
Moreover, $R$ is called balanced if $\op{R}\Gamma_{m_A}(R) \simeq A^*$ and $\op{R}\Gamma_{m_{A^\circ}}(R) \simeq A^*$ in $\op{D^b}(\op{Gr}(A^e))$.

\end{dfn}

\section{Calabi-Yau conditions}
\label{sec2}

\begin{dfn}[{\cite[Section 2.2]{kanazawa2015}}]
 Let $A$ be a connected right noetherian $\mathbb{N}$-graded $k$-algebra.
Then, $\op{proj}(A)$ is a projective Calabi-Yau $n$ scheme if the global dimension of $\op{qgr}(A)$ is $n$ and the Serre functor of the derived category $\op{D^b}(\op{qgr}(A))$ is the $n$-shift functor $[n]$. 
\end{dfn} 

\begin{rmk}
 Actually, we do not need the condition that the global dimension of $\op{qgr}(A)$ is $n$. 
If the Serre functor of the derived category $\op{D^b}(\op{qgr}(A))$ is the $n$-shift functor $[n]$, then we can easily show that this condition holds.
However, when we prove the existence of the Serre functor of  $\op{D^b}(\op{qgr}(A))$, we essentially need the condition that the global dimension of $\op{qgr}(A)$ is $n$ (cf. \cite[Theorem A.4, Corollary A.5]{de2004ideal}, Lemma \ref{duality}).
\end{rmk}
\subsection{{$\mathbb{Z}^2$}-graded algebras and Segre products}

In commutative algebraic geometry, when $X$ is the Segre embedding of $\mathbb{P}^{n} \times \mathbb{P}^{m}$ into $\mathbb{P}^{nm+n+m}$, a smooth complete intersection $Y \subset X$ of bidegrees $(n+1,0)$ and $(0,m+1)$ provides a Calabi-Yau variety.
We also have a little more complicated example that gives a Calabi-Yau variety.
That is a smooth complete intersection of bidegrees $(n,0)$ (resp. $(n+1,0)$) and $(1,n+1)$ in $\mathbb{P}^{n} \times \mathbb{P}^{n}$ (resp. $\mathbb{P}^{n+1} \times \mathbb{P}^{n}$).
We construct noncommutative analogues of these examples.


Let $C$ be an $\mathbb{N}^2$-graded $k$-algebra.
 We denote the category of $\mathbb{Z}^2$-graded right $C$-modules (resp. finitely generated $\mathbb{Z}^2$-graded right $C$-modules) by $\op{BiGr}(C)$ (resp. $\op{bigr}(C)$).
Let $M \in \op{BiGr}(C)$.
We denote by $C^\circ$ (resp. $C^e$) the opposite (resp. enveloping) algebra of $C$. 
We define the Matlis dual $M^{*} \in \op{BiGr}(C^\circ)$ by $M^*_{i,j} := \op{Hom}_k(M_{-i,-j},k)$ and the shift $M(n,m) \in \op{BiGr}(C) $ by $ M(m,n)_{i,j} := M_{i+m,j+n} \ (m, n, i, j \in \mathbb{Z})$.
For $M,N \in \op{BiGr}(C)$, we write ${\op{Hom}}_C(M,N):= \bigoplus_{m,n \in \mathbb{Z}}\op{Hom}_{\op{BiGr(C)}}(M,N(m,n))$.
For a bihomogeneous element $m \in M$, we denote the bidegree of $m$ by $\op{bideg}(m)$.

Let $M \in \op{BiGr}(C)$.
We define the truncation $M_{\geq n, \geq n} := \bigoplus_{i\geq n,j \geq n}M_{i,j} \in \op{BiGr}(C) \ (n \in \mathbb{Z})$.
We say $m \in M$ is torsion if $m C_{\geq n, \geq n} =0$ for $n \gg 0$.
If all $m \in M$ are torsion, then $M$ is called a torsion $C$-module.
We denote the category of $\mathbb{Z}^2$-graded torsion $C$-modules by $\op{Tor}(C)$.
We also define $\op{tor}(C)$ to be the intersection of $\op{bigr}(C)$ and $\op{Tor}(C)$.
When we assume that $C$ is right noetherian, we have the quotient categories $\op{QBiGr}(C):= \op{BiGr}(C) / \op{Tor}(C)$ and $\op{qbigr}(C) := \op{bigr}(C)/\op{tor}(C)$ (cf. \cite[Section 2]{rompay1996483}).
We  denote the projection functor by $\pi$ and its right adjoint functor by $\omega$.
We can define the general (resp. noetherian) projective scheme $\op{Proj}(C)$ (resp. $\op{proj}(C)$) associated to $C$ and the notion of projective Calabi-Yau schemes as in the case of $\mathbb{N}$-graded algebras.

Let $D$ be an $\mathbb{N}^2$-graded algebra.
We define $m_{C_{++}} := C_{ \geq 1, \geq 1}$ and the torsion functor $\Gamma_{m_{C_{++}}} : \op{BiGr}(C \otimes_k D^\circ) \rightarrow \op{BiGr}(C \otimes_k D^\circ)$ by $\Gamma_{m_{C_{++}}}(M) := \{m \in M \mid m C_{\geq n, \geq n} = 0 \text{ for some }n \in \mathbb{N} \}$.
We write $m_C:= \bigoplus_{i+j \geq 1 }C_{i,j}$ 
and  define another torsion functor $\Gamma_{m_C}:\op{BiGr}(C \otimes_k D^\circ) \rightarrow \op{BiGr}(C \otimes_k D^\circ)$ by $\Gamma_{m_C}(M) := \{m \in M \mid mC_{\geq n} = 0 \text{ for some } n \in \mathbb{N}\}$, where $C_{ \geq n}:= \bigoplus_{i+j \geq n} C_{i,j} \in \op{BiGr}(C)$.
See \cite[Section 3]{reyes2014} for details of $\Gamma_{m_C}$.
We write $H^i_{m_{C_{++}}}:= \op{R}^i\Gamma_{m_{C_{++}}}$ and $H^i_{m_{C}}:= \op{R}^i\Gamma_{m_{C}}$.


\begin{thm}
\label{thm1}
 Let $A := k\la x_0,\cdots, x_n \ra /(x_jx_i-q_{ji}x_ix_j)_{i,j}$, $B := k\la y_0,\cdots, y_m \ra / (y_jy_i-q'_{ji}y_iy_j)_{i,j}$ and $C := A \otimes_k B$, where $q_{ji}, q'_{ji} \in k^{\times}$ for all $i,j$. 
We regard $C$ as an $\mathbb{N}^2$-graded algebra with $\op{bideg}(x_i)=(1,0)$ and $\op{bideg}(y_i)=(0,1)$ for all $i$.
 
\begin{enumerate}
 \item 
Let $f := \sum_{i=0}^n x_{i}^{n+1} $ and $ g := \sum_{i=0}^m y_{i}^{m+1}$. 
We assume that 
\begin{inparaenum}[(i)]
\item $q_{ii} = q_{ij}q_{ji} = q_{ij}^{n+1} =1$ for all $i,j$,
\item $q'_{ii} = q'_{ij}q'_{ji} = q_{ij}'^{m+1} =1$ for all $i,j$.
\end{inparaenum}

Then, $\op{proj}(C/(f,g))$
is a projective Calabi-Yau $(n+m-2)$ scheme if and only if $\prod_{i=0}^n q_{ij}$ and $\prod_{i=0}^m q'_{ij}$ are independent of $j$, respectively.

\item
Suppose that $m=n+1$ (resp. $m=n$) and $q'_{ij}=1$ for all $i,j$. 
Let $f := \sum_{i=0}^{n} x_i^{n+1}y_i$ and $g := \sum_{i=0}^{n+1} y_i^{n+1}$ (resp. $\sum_{i=0}^{n} y_i^{n}$).
We assume that $q_{ii} = q_{ij}q_{ji} = q_{ij}^{n+1} = 1$ for all $i,j$.

Then, $\op{proj}(C/(f,g))$ is a projective Calabi-Yau $(2n-1) (\text{resp. }(2n-2))$ scheme if and only if $\prod_{i=0}^{n} q_{ij}$ is independent of $j$.

\label{thm1-2}
\end{enumerate}


\end{thm}

\begin{nota}
    For simplicity, we denote the bidegrees of $f,g$ in the theorem by $(d_0,d_1), (e_0,e_1)$, respectively.
\end{nota}


\begin{rmk}
\begin{itemize}
\item $f,g$ are central elements in $C$ because of the choices of $\{q_{ij}\}, \{q'_{ij}\}$.
\item We have $n+m-2 = d_0+d_1+e_0+e_1-4$ in (1). 
We have $2n-1 (\text{resp. } 2n-2) = d_0+d_1+e_0+e_1-4$ also in (2).

\item In (2) of the theorem, even if we do not assume $q'_{ij} = 1$, the condition for $f,g$ to be central in $C$ implies $q_{ij}'=1$ for all $i,j$ after all.


\end{itemize}

\end{rmk}


To prove the theorem, we need to show some lemmas.
Perhaps some experts may understand the following lemmas.
However, to the best of the author's knowledge, there are no references written on those lemmas, so the proofs are given below.
In addition, the following proofs do not depend on whether (1) or (2) in the theorem is considered (except for Lemma \ref{lem1-2}).


\begin{lemm} 
Let $\mathcal{R}:=\pi(\op{R}\Gamma_{m_{{C/(f,g)}_{++}}}(C/(f,g))^{*})$ and $ \mathcal{R}' := \pi(\op{R}\Gamma_{m_{C/(f,g)}}(C/(f,g))^*)$.
Then, the functors $-\otimes^{\mathbb{L}}\mathcal{R}$ and $ -\otimes^{\mathbb{L}}\mathcal{R}'[-1]$ between $\op{D(QBiGr}(C/(f,g)))$ and itself are naturally isomorphic.
\label{lem1-1}
\end{lemm}

\begin{proof}
Let $I_1,I_2$ be the ideals generated by $\{x_0,\cdots,x_n \}, \{y_0\cdots,y_m\}$, respectively.
Then, we have $m_{{C/(f,g)}_{++}} = I_1 \cap I_2$, $m_{C/(f,g)} = I_1 + I_2$ and have the following long exact sequence in $\op{BiGr}(C/(f,g)^e)$
{\footnotesize
\[
 \cdots \rightarrow H^i_{m_{C/(f,g)}}(C/(f,g)) \rightarrow H^i_{I_1}(C/(f,g)) \oplus H^i_{I_2}(C/(f,g)) \rightarrow H^i_{m_{{C/(f,g)}_{++}}}(C/(f,g)) \rightarrow  \cdots
\]}
 by using the Mayer-Vietoris sequence, where $\Gamma_{I_j} (j=1,2)$ is defined not by using the degrees of $I_j$ but by using powers of $I_j$ (i.e., $\Gamma_{I_j}(M) := \{m \in M \mid mI_j^n  = 0 \text{ for some }n\}$).
Note that we can use the Mayer-Vietoris sequence in our case because $I_1, I_2$ are generated by normal elements and this implies that $I_1, I_2$ satisfy Artin-Rees property. 
We also have the exact triangle in $\op{D}(\op{BiGr}(C/(f,g)^e))$
{\small
\[
 \op{R}\Gamma_{m_{C/(f,g)}}(C/(f,g)) \rightarrow \op{R}\Gamma_{I_1}(C/(f,g))\oplus \op{R}\Gamma_{I_2}(C/(f,g)) \rightarrow \op{R}\Gamma_{m_{{C/(f,g)}_{++}}}(C/(f,g)).
\]
}
Moreover, $H^i_{I_1}(C/(f,g))^*$ and $H^i_{I_2}(C/(f,g))^*$ are torsion modules for $m_{C/(f,g)_{++}}$ from Sub-Lemma \ref{basechange}.
So, the cohomologies of $\op{R}\Gamma_{I_1}(C/(f,g))^* \oplus \op{R}\Gamma_{I_2}(C/(f,g))^*$ are torsion.
Combining this result with the above triangle, we get the claim.

\end{proof}



\begin{slem}
Let $I_1,I_2$ be as in the proof of Lemma \ref{lem1-1}.
$H^i_{I_1}(C/(f,g))^{*}$ and $H^i_{I_2}(C/(f,g))^{*}$ are torsion modules for $m_{{C/(f,g)_{++}}}$ for any $i$.
 \label{basechange}
\end{slem}

\begin{proof}
We only show that $H^i_{I_1}(C/(f,g))^{*}$ are torsion modules for $m_{C/(f,g)_{++}}$.
We can show that $H^i_{I_2}(C/(f,g))^{*}$ are torsion in the same way.

First, we prove that $H^i_{I_1}(C)^{*}$ is torsion.
We have $\Gamma_{I_1} = \Gamma_{I_1^{n+1}}$.
Moreover, if $J_1$ is the ideal generated by $x_0^{n+1},\cdots, x_n^{n+1}$, then we have $\Gamma_{I_1^{n+1}} = \Gamma_{J_1}$.
Note that $x_0^{n+1},\cdots, x_n^{n+1}$ are central elements in $C$ from the choice of $\{q_{ij}\}$.

Let $M \in \op{Gr}(C)$ be injective.
Then, we have a surjective localization map $M \rightarrow M[x_i^{-(n+1)}]$ for any $i$ and $\Gamma_{J_1}(M)$ is injective in $\op{Gr}(C)$ because $J_1$ satisfies Artin-Rees property (cf. \cite[Lemma A1.4]{eisenbud2005geometry}). 
When $M'$ is injective in $\op{Gr}(C^e)$, then $M'$ is injective in $\op{Gr}(C)$,
where $\op{Res}_{C}: \op{Gr}(C^e) \rightarrow \op{Gr}(C)$ is the restriction functor (\cite[Lemma 2.1]{yekutieli1992}).
Thus, we can calculate $\op{Res}_{C}(H^i_{J_1}(C))$ by using a $\check{\op{C}}$ech complex $\mathscr{C}(x_0^{n+1},\cdots, x_n^{n+1} ; C)$ (cf. \cite[Theorem A1.3]{eisenbud2005geometry}, \cite[Chapter 2, 3]{mckemey2012local}).
Then, we have $\mathscr{C}(x_0^{n+1},\cdots, x_n^{n+1} ; C) = \mathscr{C}(x_0^{n+1},\cdots, x_n^{n+1} ; A) \otimes_k B$. 
This induces that $\op{Res}_C({H^i_{J_1}(C)}) \simeq H^i_{m_A}(A) \otimes_k B$.
Because $H^i_{m_A}(A)_{>0} = 0$ (\cite[Proposition 2.4]{kanazawa2015}), $H^i_{I_1}(C)^{*}$ is torsion. 

Finally, we consider the exact sequences of $C$-bimodules
\begin{align}
 &0 \rightarrow C (-d_0, -d_1) \overset{\times f}{\rightarrow} C\rightarrow C/(f) \rightarrow 0,\label{234132_12Sep23} \\
 &0 \rightarrow C/(f)(-e_0, -e_1) \overset{\times g}{\rightarrow} C/(f) \rightarrow C/(f,g) \rightarrow 0.\label{234152_12Sep23}
\end{align}
Then, we take the long exact sequence for $\Gamma_{I_1}$ and 
 we get the claim since $H^i_{I_1}(C)^*$ is torsion.
\end{proof}

\begin{lemm}
 $\op{gl.dim}(\op{qbigr}(C/(f,g))) = d_0+d_1+e_0+e_1-4$.
\label{lem1-2}
\end{lemm}

\begin{proof}
We show the proposition only in (1) of the theorem.
In (2) of the theorem, the proposition can be shown in the same way (cf. Remark \ref{192351_30Aug23}).
 We consider a bigraded (commutative) algebra $D := k[s_0,\cdots,s_n ,t_0,\cdots, t_m]/(\sum_{i=0}^ns_i,\sum_{i=0}^mt_i)$ with $s_i=x_i^{n+1},t_i=y_i^{m+1}$ and the projective spectrum $\op{biProj}(D)$ in the sense of \cite[Section 1]{hyry1999}.
Then, $C/(f,g)$ is a finite $D$-module.
So, $\op{qbigr}(C/(f,g))$ can be thought of as the category of modules over a sheaf $\mathcal{A}$ of $\mathcal{O}_{\op{biProj}(D)}$-algebras, where $\mathcal{A}$ is the sheaf on $\op{biProj}(D)$ which is locally defined by the algebra $(k[x_0,\cdots,x_n, y_0,\cdots,y_m]/(f,g)_{x_iy_j})_{(0,0)}$ on each open affine scheme $\op{D}_{+}(s_it_j) \simeq \op{Spec}((D_{s_it_j})_{(0,0)})$.
Hence, it is enough to prove that the global dimension of $(k[x_0,\cdots,x_n, y_0,\cdots,y_m]/(f,g)_{x_iy_j})_{(0,0)} = d_0+d_1+e_0+e_1-4 = n+m-2$.

We can complete the rest of the proof in the same way as in \cite[Section 2.3]{kanazawa2015}. 
We give its sketch.
For simplicity, we prove the claim when $i=j=0$.
We define a $k$-algebra $E$ by
$E:=k[S_1,\cdots,S_n,T_1,\cdots,T_m]/(1+\sum_{i=1}^n S_i, 1+\sum_{i=0}^mT_i)$ with $S_i = s_i/s_0, T_i = t_i/t_0$.
We also define an $E$-algebra $F$ by 
    $F := k \langle X_1,\cdots, X_n, Y_1,\cdots, Y_m \rangle/(X_iX_j - (q_{0i}q_{ij}q_{j0}) X_jX_i, Y_iY_j- (q'_{0i}q'_{ij}q'_{j0})Y_jY_i, 1+\sum_{i=1}^n X_{i}^{n+1}, 1+\sum_{i=1}^mY_{i}^{m+1})$ with $X_i = x_i/x_0, Y_i =y_i/y_0$.
The module structure of $F$ is given by the identifications $S_i = X_i^{n+1}, T_i = Y_i^{m+1}$.
Let $F_{\tilde{m}}$ be the localization of $F$ at a maximal ideal $\tilde{m} := (S_1 - a_1, \cdots, S_n- a_n, T_1 - b_1, \cdots, T_m - b_m) $ of $E$ with $1+\sum_{i=1}^{n}a_i=1+\sum_{i=1}^mb_i=0 \ (a_i,b_i \in k)$.
Then, it is enough to prove that the global dimension of $F_{\tilde{m}}$ is $n+m-2$ (\cite[Lemma 2.6, 2.7]{kanazawa2015}).

If all $a_i,b_i$ are not $0$, then $F/\tilde{m}F$ is a twisted group ring and hence semisimple.
Moreover, $S_1 - a_1, \cdots, S_n- a_n, T_1 - b_1, \cdots, T_m - b_m$ is a regular sequence in $F_{\tilde{m}}$.
This induces the claim (\cite[Theorem 7.3.7]{mcconnell2001}).

On the other hand, we assume that one of $\{a_1, \cdots, a_n, b_1, \cdots, b_m\}$ is $0$. 
For example, we assume $a_1 = 0$.
We consider $F/(X_1)$.
Then, we can show that the global dimension of $(F/(X_1))_{\tilde{m}} = n+m-3$ because $\op{pd}_{F}(S) = \op{pd}_{F/(X_1)}(S) +1$ for any simple $F$-module $S$ with $\op{Ann}(S)= \tilde{m}$ (\cite[Theorem 7.3.5]{mcconnell2001}). 
If some other $a_i, b_j$ are $0$, we repeat taking quotients and can reduce to considering the global dimension of the algebra $k[X,Y]/(X^{n+1}+1,Y^{m+1}+1)$, which are $0$.
\end{proof}

\begin{rmk}
\label{192351_30Aug23}

To prove Lemma \ref{lem1-2} in (2) of the theorem, consider the projective spectrum $X:=\op{biProj}(k[s_0,\cdots,s_{n}, t_0,\cdots,t_{n+1}]/$ $(\sum_{i=0}^{n} s_i t_i, \sum_{i=0}^{n+1} t_i^{n+1} ))$ (resp. $\op{biProj}(k[s_0,\cdots,s_{n}, t_0,\cdots,t_n]/$ $(\sum_{i=0}^{n} s_i t_i, \sum_{i=0}^{n} t_i^{n} ))$) and the sheaf $\mathcal{A}$ of algebras on $X$ associated to $C/(f,g)$. 
\end{rmk}

\begin{proof}[Proof of Theorem \ref{thm1}]
\label{115802_31Aug23}
First, we calculate $\op{R}\Gamma_{m_{C/(f,g)}}(C/(f,g))^{*}$.
From \cite[Proposition 2.4]{kanazawa2015} (or \cite[Example 5.5]{reyes2014}) and the proof of \cite[Lemma 6.1]{reyes2014}, we have 
\begin{align*}
  \op{R}\Gamma_{m_{C}}(C)^* &\simeq \op{R}\Gamma_{m_{A}}(A)^* \otimes \op{R}\Gamma_{m_{B}}(B)^*  \\
&\simeq {}^\phi A^1 (-d_0-e_0) \otimes_k {} ^\psi B^1 (-d_1-e_1)[d_0+d_1+e_0+e_1],
\end{align*}
where $\phi$ (resp. $\psi$) is the graded automorphism of $A$ (resp. $B$) which maps $x_j \mapsto \prod_{i=0}^{n} q_{ji}x_j$ (resp. $y_j \mapsto \prod_{i=0}^m q_{ji}'y_j$).
Then, we consider the distinguished triangles 
\begin{align*}
 &\op{R}\Gamma_{m_C}(C(-d_0,-d_1)) \overset{\times f}{\longrightarrow} \op{R}\Gamma_{m_C}(C) \longrightarrow \op{R}\Gamma_{m_{C/(f)}}(C/(f)), \\
 &\op{R}\Gamma_{m_{C/(f)}}((C/(f))(-e_0,-e_1)) \overset{\times g}{\longrightarrow} \op{R}\Gamma_{m_{C/(f)}}(C/(f)) \longrightarrow \op{R}\Gamma_{m_C/(f,g)}(C/(f,g))
\end{align*}
obtained from the exact sequences \ref{234132_12Sep23} and \ref{234152_12Sep23} of $C$-bimodules.
Hence, we have 
\begin{align}
\label{formula}
 \op{R}\Gamma_{m_{C/(f,g)}}(C/(f,g))^{*} \simeq  {}^{\phi \otimes \psi} (A \otimes_k B /(f,g))^1[d_0+d_1+e_0+e_1-2].
\end{align}
In addition, we have the Serre duality in $\op{D^{b}}(\op{qbigr}(C/(f,g)))$ from Lemma \ref{dulaitylem}.
Thus, $- \otimes^{\mathbb{L}} \pi(\op{R}\Gamma_{m_{C/(f,g)_{++}}}(C/(f,g))^{*})[-1]$ is the Serre functor of $\op{D^{b}}(\op{qbigr}(C/(f,g)))$ because this functor induces an equivalence from Lemma \ref{lem1-1} and the formula \ref{formula}.
Finally, the Serre functor $- \otimes^{\mathbb{L}} \pi(\op{R}\Gamma_{m_{C/(f,g)_{++}}}(C/(f,g))^{*})[-1]$ induces the $[d_0+d_1+e_0+e_1-4]$-shift functor if and only if $\prod_{i=0}^n q_{ij}$ and $\prod_{i=0}^m q'_{ij}$ are independent of $j$ (cf. \cite[Remark 2.5]{kanazawa2015}). This completes the proof.
 
\end{proof}

The following lemma is well-known in the case of $\mathbb{N}$-graded algebras (for example, see \cite{de2004ideal}, \cite{yekutieli1997serre}).

\begin{lemm}[Local Duality and Serre Duality for $\mathbb{N}^2$-graded algebras]
\label{duality}
 Let $D$ be a connected right noetherian $\mathbb{N}^2$-graded $k$-algebra (connected means $D_{0,0}=k$).
 Let $E$ be a connected $\mathbb{N}^2$-graded $k$-algebra.
We assume that $\Gamma_{m_{D_{++}}}$ has finite cohomological dimension.
\begin{enumerate}
\item Let $Q := \omega \circ \pi: \op{BiGr}(D) \rightarrow \op{BiGr}(D)$.
 Let $M \in \op{D}(\op{BiGr}(D \otimes_k E^{\circ}))$.
Then,
\begin{align}
 \op{R}\Gamma_{m_{D_{++}}}(M)^* &\simeq  \op{R}{\op{Hom}}_{D}(M, \op{R}\Gamma_{m_{D_{++}}}(D)^*), \label{ld1} \tag{a} \\ 
     \op{R}Q(M)^* &\simeq  \op{R}{\op{Hom}}_{D}(M, \op{R}Q(D)^*) \label{ld2} \tag{b}
\end{align}
in $\op{D}(\op{BiGr}(D \otimes_k E^{\circ}))$, where we denote the natural extension of $Q$ to a functor between $\op{BiGr}(D \otimes_k E^{\circ})$ and itself by the same notation.
\item  We assume that $\op{qbigr}(D)$ has finite global dimension.
Let $\mathcal{M} := \pi(M)$, $\mathcal{N} := \pi(N)$ $(M,N \in \op{D^b}(\op{bigr}(D)))$.
Suppose ${\mathcal{R}}_D := \pi(\op{R}\Gamma_{m_{D_{++}}}(D)^*) \in \op{D^b}(\op{qbigr}(D))$. 
Then, $\mathcal{N} \otimes^\mathbb{L} {\mathcal{R}}_D \in \op{D^b}(\op{qbigr}(D))$ and 
\[
 \op{Hom}_{\op{D^b}(\op{qbigr}(D))}(\mathcal{N},\mathcal{M}) \simeq \op{Hom}_{\op{D^b}(\op{qbigr}(D))}(\mathcal{M}, (\mathcal{N} \otimes^\mathbb{L} {\mathcal{R}}_D)[-1])',
\]
which is functorial in $\mathcal{M}$ and $\mathcal{N}$.
Here, $(-)'$ denotes the $k$-dual. 
\end{enumerate}
\label{dulaitylem}
\end{lemm}

\begin{proof}
Since $\op{R}^i\Gamma_{m_{D_{++}}}(-) \simeq \lim_{n \to \infty} \op{Ext}^i(D/D_{\geq n, \geq n},-)$ and $D$ is right noetherian, one can check that $\op{R}^i\Gamma_{m_{D_{++}}}(-)$ commutes with direct limits as in  \cite[Proposition 16.3.19]{yekutieli2019derived}.
In addition, if $K$ is a complex of graded free right $D$-modules and $L$ is a complex of graded right $D^e$-modules, then $\Gamma_{m_{D_{++}}}(K \otimes_D L) \simeq K \otimes_D \Gamma_{m_{D_{++}}}(L)$ (cf. \cite[Lemma 6.10]{mori2021}).
So, we can apply the argument of \cite[Theorem 5.1]{vandenbergh1997} (or \cite[Theorem 2.1]{mori2024corrigendum}) to prove (\ref{ld1}) of (1).

In order to prove (\ref{ld2}) of (1), note that we have the canonical exact sequence and the isomorphism (see also \cite[Lemma 4.1.4, 4.1.5]{bondal2002generators})
\begin{gather*}
 0 \rightarrow \Gamma_{m_{D_{++}}}(M) \rightarrow M \rightarrow Q(M) \rightarrow \lim_{n \to \infty} \op{Ext}^1(D/D_{\geq n, \geq n}, M) \rightarrow 0 , \\
 \op{R}^iQ(M) \simeq \op{R}^{i+1}\Gamma_{m_{D_{++}}}(M), \quad (1 \leq i, M \in \op{BiGr}(D)).
\end{gather*}
So, from the previous paragraph, $Q$ 
 has finite cohomological dimension, $\op{R}^iQ$ commutes with direct limits.
 We also have $Q(K \otimes_D L) \simeq K \otimes_D Q(L)$, where $K,L$ are as above (cf. \cite[Lemma 3.28]{mori2023categorical}).
 Hence, we can also apply the argument of \cite[Theorem 5.1] {vandenbergh1997} (or \cite[Theorem 3.29]{mori2023categorical}) to prove (\ref{ld2}) of (1).

We can prove (2) in the same way as in \cite[Lemma A.1, Theorem A.4]{de2004ideal} by using (\ref{ld2}) of (1).
Note that we have a natural equivalence $\op{D^b}(\op{qbigr}(D)) \simeq \op{D^b_f}(\op{QBiGr}(D))$, where $\op{D^b_f}(\op{QBiGr}(D))$ is the full subcategory of $\op{D^b}(\op{QBiGr}(D))$ consisting of complexes with cohomology in $\op{qbigr}(D)$ (\cite[Lemma 2.2]{de2004ideal}).

\end{proof}

As a corollary of Theorem \ref{thm1}, we construct examples of noncommutative projective Calabi-Yau schemes by using Segre products.
Let $A,B,f$ and $g$ be as in Theorem $\ref{thm1}$.
\begin{dfn}
\label{sgr-pd}
\begin{enumerate}
 \item  The Segre product $A \circ B$ of $A$ and $B$ is the $\mathbb{N}$-graded $k$-algebra with $(A \circ B)_i = A_i \otimes_k B_i$.
\item Let $M \in \op{bigr}(C)$ .
We define a right graded $A \circ B$-module $M_\Delta$ as the graded $A \circ B$-module with $(M_{\Delta})_i = M_{i,i}$.
\end{enumerate}
\end{dfn}

\begin{lemm}[{\cite[Theorem 2.4]{rompay1996483}}]
We have the following natural isomorphism
\[
\xymatrix@R=5pt{
   \op{qbigr}(C) \ar[r] &  \op{qgr}(A \circ B), & 
 \pi(M) \ar@{|->}[r] & \pi(M_\Delta).
}
\]
 In addition, the functor defined by $- \otimes_{A \circ B}C$ is the inverse of this equivalence.
\label{lemrompay}
\end{lemm}

\begin{rmk}
Let $J := (f,g) \in \op{bigr}(C)$.
We similarly obtain an equivalence
\[
   \op{qbigr}(C/J) \simeq  \op{qgr}(A \circ B/J_\Delta).
\]\label{rmkrompay}
\end{rmk}
\vspace{-1cm}
Combining Theorem \ref{thm1} with Remark \ref{rmkrompay}, we get the following.
\begin{cor}
 Let $J := (f,g) \in \op{bigr}(C)$.
Then, $\op{proj}(A \circ B/ J_\Delta)$ is a projective Calabi-Yau scheme.
\label{cor1}
\end{cor}

\subsection{Weighted hypersurfaces}
\label{122542_31Aug23}
Reid produced the list of all commutative weighted Calabi-Yau hypersurfaces of dimensions 2 (for example, see  \cite{iano2000}, \cite{reid1980}).
In this section, we construct noncommutative projective Calabi-Yau schemes from noncommutative weighted projective hypersurfaces.
Let $A$ be a right noetherian $\mathbb{N}$-graded $k$-algebra.
Then, the $r$-th Veronese algebra $A^{(r)}$ is the $\mathbb{N}$-graded $k$-algebra with $A^{(r)}_i = A_{ri}$.
We consider the (commutative) weighted polynomial ring $A = k[x_0,\cdots,x_n]$ with $\op{deg}(x_i) = d_i$.
Then, $\op{Coh}(\op{Proj}(A))$ is in general not equivalent to $\op{qgr}(A)$, but to $\op{qgr}(A^{(n+1)\op{lcm}(d_0,\cdots,d_n)})$.
However, we can think of $\op{qgr}(A)$ as a resolution of singularities of $\op{Coh}(\op{Proj}(A))$ (cf. \cite[Example 4.9]{smith2004maps}).
Moreover, we have $\op{qgr}(A) \simeq \op{Coh}([(\op{Spec}(A) \backslash \{ 0\}) / \mathbb{G}_m])$ and $[(\op{Spec}(A) \backslash \{ 0\}) / \mathbb{G}_m]$ is a smooth Deligne-Mumford stack whose coarse moduli space is $\op{Proj}(A)$.

\begin{thm}
 Let $(d_0,\cdots, d_n) \in \mathbb{Z}_{>0}^{n+1}$ and $d := \sum_{i=0}^n d_i$ such that $d$ is divisible by $d_i$ for all $i$.
Let $C:= k\la x_0,\cdots,x_n \ra / (x_jx_i-q_{ji}x_ix_j)_{i,j}$, where $q_{ji} \in k^\times, \op{deg}(x_i) = d_i$ for all $i,j$.  
Let $f := \sum_{i=0}^n x_i^{h_i}$, where $h_i := d/d_i$.

We assume that $q_{ii} =q_{ij}q_{ji} = q_{ij}^{h_i} = q_{ij}^{h_j} =1$ for all $i,j$.
Then, $\op{proj}(C/(f))$ is a projective Calabi-Yau $(n-1)$ scheme if and only if there exists $c \in k$ such that $c^{d_j} = \prod_{i=0}^{n}q_{ij}$ for all $j$.
\label{thm2}
\end{thm}

\begin{rmk}
\begin{itemize}
 \item $f$ is a central element in $C$ from the choice of $\{q_{ij}\}$.
\item Theorem \ref{thm2} is a generalization of \cite[Theorem 1.1]{kanazawa2015}.
\end{itemize}

\end{rmk}


\begin{lemm}
 The balanced dualizing complex of $C/(f)$ is isomorphic to ${^\phi}(C/(f))^1[n]$, where $\phi$ is a graded automorphism of $C$ which maps $x_j \mapsto \prod_{i=0}^n q_{ji}x_j$.
\label{lem3}
\end{lemm}

\begin{proof}
Since $C$ is Artin-Schelter regular, $C$ is skew Calabi-Yau (\cite[Lemma 1.2]{reyes2014}).
This induces that the balanced dualizing complex of $C$ is isomorphic to ${}^\phi C^1(-d)[n+1]$, where $\phi$ is the Nakayama automorphism of $C$.
From \cite[Example 5.5]{reyes2014}, the automorphism $\phi$ is the map which maps $x_j \mapsto \prod_{i=0}^n q_{ji}x_j$.

By using this result, we can obtain the claim in the same way as in the proof of Theorem \ref{thm1} after Remark \ref{192351_30Aug23}.
\end{proof}

To calculate the global dimension of $\op{qgr}(C/(f))$, we recall the notion of quasi-Veronese algebras.
In detail, see \cite[Section 3]{mori2013b}.
\begin{dfn}[{\cite[Section 3]{mori2013b}}]

Let $A$ be an $\mathbb{N}$-graded $k$-algebra.
The $l$-th quasi-Veronese algebra $A^{[l]}$ of $A$ is a graded $k$-algebra defined by 
\[
 A^{[l]} :=\bigoplus_{i \in \mathbb{N}} A^{[l]}_i  := \bigoplus_{i \in \mathbb{N}} 
\begin{pmatrix}
 A_{li} & A_{li+1} & \cdots & A_{li+l-1} \\
A_{li-1} & A_{li} & \cdots & A_{li+l-2} \\
\vdots & \vdots & \ddots & \vdots \\
A_{li-l+1} & A_{li-l+2} & \cdots & A_{li}.
\end{pmatrix}.
\]
 
\end{dfn}

\begin{rmk}
\label{125114_1Sep23}
\begin{enumerate}
    \item We have $\op{Gr}(A) \simeq \op{Gr} (A^{[l]})$ (\cite[Lemma 3.9]{mori2013b}).
The equivalence is obtained by the functor $Q:\op{Gr}(A) \rightarrow \op{Gr}(A)$, which is defined by $Q(M) := \bigoplus_{i \in \mathbb{Z}} \left( \bigoplus_{j=0}^{l-1} M_{li-j} \right)$
\item When $A$ is right noetherian, $A^{[l]} \simeq \bigoplus_{0 \leq i,j \leq n-1} A(j-i)^{(l)} \in \op{gr}(A^{(l)})$, where $A^{(l)}$ is the $l$-th Veronese algebra of $A$ and the $A^{(l)}$-module structure of $A^{[l]}$ is given by the natural inclusion $A^{(l)} \subset A^{[l]}$ (cf. the proof of \cite[Proposition 4.11]{mori2015regular}). 
Then, $A^{[l]}$ is also right noetherian since $A^{(l)}$ is right noetherian.
In this case, $Q$ induces an equivalence between $\op{qgr}(A)$ and $\op{qgr}(A^{[l]})$. 
 
\end{enumerate}

\end{rmk}

\begin{lemm}
\label{190844_31Aug23}
 Let $A$ be an $\mathbb{N}$-graded $k$-algebra which is generated by homogeneous elements $y_0, \cdots, y_h$ with $\op{deg}(y_i) >0 $ as an $A_0$-algebra.
Let $l \geq \op{max}\{\op{deg}(y_0), \cdots, \op{deg}(y_h)\}$.
Then, $A^{[l]}$ is generated in degree $0$ and $1$.
\end{lemm}

\begin{proof}
For any $i \in \mathbb{N}$ and any $a,b \in \{0,1, \cdots, l-1 \}$, it is enough to show that every homogeneous element $m$ of the form
{\small
\[
m=
 \begin{pmatrix} 
  m_{0,0} & \dots & m_{0,\beta} & \dots & m_{0,l-1} \\
  \vdots &      & \vdots &       & \vdots \\
  m_{\alpha,0} & \dots & m_{\alpha,\beta} & \dots & m_{\alpha,l-1} \\
 \vdots &       & \vdots &     & \vdots \\
  m_{l-1,0} & \dots & m_{l-1,\beta} & \dots & m_{l-1,l-1}
\end{pmatrix} \in  A^{[l]}_i, \quad
\left( \begin{gathered}
	m_{\alpha, \beta} \in  \left( A^{[l]}_i \right)_{\alpha,\beta}:= A_{li+\beta-\alpha}, \\
        \  m_{\alpha, \beta} = 0 \text{ when } (\alpha,\beta) \neq (a,b) \\
        0 \leq \alpha, \beta \leq l-1
       \end{gathered}
\right)
\]
}
is generated in degree $0$ and $1$.
Moreover, we can assume that $m_{a,b} = \prod_{j=0}^{n_1} y_{i_j} \ (i_j \in \{0,\cdots,h\}, n_1 \in \mathbb{N})$.

If $m_{a,b}$ is decomposed into $\prod_{j=0}^{n_1} y_{i_j} = \prod_{j=0}^{n_2} y_{i_j} \prod_{j=n_2+1}^{n_1} y_{i_j} \ (n_2 \in \mathbb{N})$ such that $l-a \leq  \op{deg}(\prod_{j=1}^{n_2} y_{i_j}) \leq 2l-a-1$, then we have $\prod_{j=0}^{n_2} y_{i_j} \in  (A^{[l]}_1)_{a,c}=  A_{l+c-a}$ and $\prod_{j=n_2+1}^{n_1} y_{i_j} \in  (A^{[l]}_{i-1})_{c,b}=  A_{l(i-1)+b-c}$ $(0 \leq {}^\exists c \leq l-1)$.
In this case, we can show the claim by using induction on the degree of $m$.
So, it is sufficient to show that we have such a decomposition for all $m$.
Indeed, we can find at least one such decomposition from $(2l-a-1)-(l-a)+1 = l$ and the choice of $l$. In detail, we have $l-a \leq \op{deg}(y_{i_0}) \leq 2l-a-1$ or there exists $n_3 \in \mathbb{N}$ such that $\op{deg}(y_{i_0}y_{i_1} \cdots y_{i_{n_3}}) < l-a$ and $l-a \leq \op{deg}(y_{i_0}y_{i_1} \cdots y_{i_{n_3}}y_{i_{n_3}+1}) \leq 2l-a-1$ since $0 < \op{deg}(y_i) \leq l$.
\end{proof}

\begin{lemm}
 $\op{gl.dim}(\op{qgr}(C/(f))) = n-1$.
\label{lem4}
\end{lemm}

\begin{proof}
We use the idea of the proof in Lemma \ref{lem1-2}.
We consider an $\mathbb{N}$-graded $k$-algebra $B := k[s_0,\cdots,s_n]/(\sum_{i=0}^n s_i)$ with $s_i=x_i^{h_i}$.
Then, $A^{[d]}$ is right noetherian and $\op{qgr}(C/(f)) \simeq \op{qgr}((C/(f))^{[d]})$ from Remark \ref{125114_1Sep23}.
So, it is enough to prove that $\op{gl.dim}(\op{qgr}((C/(f))^{[d]})) = n-1$.
Because $(C/(f))^{(d)}$ is finite over $B$, $(C/(f))^{[d]}$ is also finite over $B$.
In addition, $(C/(f))^{[d]}$ is generated in degrees $0$ and $1$ from Lemma \ref{190844_31Aug23}.
So, $\op{qgr}((C/(f))^{[d]})$ is equivalent to the category of coherent modules over a sheaf $\mathcal{A}$ of $\mathcal{O}_{\op{Proj}(B)}$-algebra, where $\mathcal{A}$ is the sheaf on the projective spectrum $\op{Proj}(B)$ which is locally defined by a tiled matrix algebra
{\small
\[N_i = 
 \begin{pmatrix}
 E_{i,0} & E_{i,1} & \cdots & E_{i,d-1} \\
E_{i,-1} & E_{i,0} & \cdots & E_{i,d-2} \\
\vdots & \vdots & \cdots &\vdots \\
E_{i,-d+1} & E_{i,-d+2} & \cdots & E_{i,0} \\
\end{pmatrix}
\]
}
 on each $D_+(s_i)$.
Here, $E_i := (C/(f))[x_i^{-1}]$ and $E_{i,j}$ is the degree $j$ part of $E_i$.
As in the proof of Lemma \ref{lem1-2}, it is enough to show that the global dimension of $N_i$ is $n-1$ for all $i$. 

On the other hand, $R_1 := E_i \oplus E_i(1) \oplus \cdots \oplus E_i(d-2) \oplus E_i(d-1)$ and $R_2 := E_i \oplus E_i(1) \oplus \cdots \oplus E_i(d_i-2) \oplus E_i(d_i-1)$ are progenerators in $\op{Gr}(E_i)$.
So, the category of right $\op{End}_{\op{gr}}(R_1)$-modules and the category of right $\op{End}_{\op{gr}}(R_2)$-modules are equivalent because they are equivalent to the category of graded right $E_i$-modules (cf. \cite[Lemma 4.8]{smith-noncommutative}, \cite[Remarks after Proposition 4.5]{smith2004maps}).
We also have $\op{End}_{\op{gr}}(R_1) \simeq N_i$ and  
{\small
\[\op{End}_{\op{gr}}(R_2) \simeq M_i := 
 \begin{pmatrix}
 E_{i,0} & E_{i,1} & \cdots & E_{i,d_i-1} \\
E_{i,-1} & E_{i,0} & \cdots & E_{i,d_i-2} \\
\vdots & \vdots & \cdots &\vdots \\
E_{i,-d_i+1} & E_{i,-d_i+2} & \cdots & E_{i,0} \\
\end{pmatrix}.
\]
}
So, it is sufficient to prove the global dimension of $M_i$ is $n-1$ for each $i$.

For simplicity, we assume $i=0$.
When $i \neq 0$, we can show the claim in the same way.
Let $D=k[S_1,\cdots, S_n]/(1+\sum_{j=0}^n S_j)$ with $S_j = s_j/s_0$.
We show that the global dimension of the $D$-algebra $M_0$ is $n-1$.
The module structure of $M_0$ is given by the identification $S_j = (x_j^{h_j}/x_0^{h_0})I_{d_0} \in M_0$, where $I_{d_0}$ is the $(d_0 \times d_0)$-identity matrix.
Let $\tilde{m} = (S_1-a_1, \cdots S_n-a_n) \  (a_j \in k)$ be a maximal ideal of $D$ with $1+ \sum_{j=1}^n a_j = 0$.
It is sufficient to show that $\op{gl.dim}((M_0)_{\tilde{m}})=n-1$, where $(M_0)_{\tilde{m}}$ is the localization of $M_0$ at $\tilde{m}$ (cf. the second paragraph of the proof of Lemma \ref{lem1-2}).
We divide the proof of this claim into two cases.

\textit{Case (a) : all $a_j$ are not $0$. } 
Because $S_1-a_1,\cdots,S_n-a_n$ is a regular sequence in $(M_0)_{\tilde{m}}$, we show that the global dimension of $(M_0)_{\tilde{m}}/ \tilde{m} (M_0)_{\tilde{m}} \simeq M_0 / \tilde{m} M_0$ is $0$ (cf. the third paragraph of the proof of Lemma \ref{lem1-2}).

First, the category of $M_0 / \tilde{m} M_0$-modules is equivalent to the category of graded $E_0' := E_0/(x_1^{h_1}/x_0^{h_0}-a_1,\cdots, x_n^{h_n}/x_0^{h_0}-a_n)E_0$-modules.
This is a Morita equivalence obtained from the isomorphism $\op{End}_{\op{gr}}(E_0') \simeq M_0 / \tilde{m} M_0$ (cf. the three previous paragraph).

Next, we see that $E_0'$ is strongly graded.
Since $E_0 \simeq ( C[x_0^{-1}])/ (1 + (x_1^{h_1}/x_0^{h_0}) + \cdots + (x_n^{h_n}/x_0^{h_0}))$, we have $E_0' \simeq (C[x_0^{-1}])/ (x_1^{h_1}/x_0^{h_0}-a_1,\cdots, x_n^{h_n}/x_0^{h_0}-a_n)$.
For any $l \in \mathbb{Z}$, if $\tilde{x} := x_0^{l_0}x_1^{l_1} \cdots x_n^{l_n} \in (E_0')_{l} \ (l_0 \in \mathbb{Z}, l_1, \cdots l_n \in \mathbb{N})$, then there exist $k_1,\cdots,k_n \in \mathbb{N}$ such that $\tilde{x}' :=  x_0^{(-\sum k_i)h_0 - l_0}x_1^{k_1h_1-l_1} \cdots x_n^{k_nh_n-l_n} \in (E_0')_{-l}$.
Because $\tilde{x} \tilde{x}' \in k^{*}$, we get $1 \in (E_0')_l(E_0')_{-l}$ and $E_0'$ is strongly graded.

Since $E_0'$ is strongly graded, we have $\op{Gr}(E_0') \simeq \op{Mod}((E_0')_0)$.
Then, $(E_0')_0$ is a twisted group algebra, where a $k$-basis of $(E_0')_0$ is $\{x_0^{e_0}x_1^{e_1}x_2^{e_2} \cdots x_n^{e_n} \in (E_0')_0\mid \sum_{j=0}^n e_jd_j =0   \text{ and }  0 \leq e_j < h_j \ ({}^\forall j =1,2,\cdots,n) \}$.
In particular, $(E_0')_0 $ is semisimple.
Hence, the graded global dimension of $E_0'$ is $0$ and  $\op{gl.dim}(M_0/\tilde{m}M_0)=0$. 

\textit{Case (b) : some of $a_j$ are $0$. }
For example, we assume $a_1 = 0$.
Then, $(x_1^{h_1}/x_0^{h_0})I_{d_0}$ is an annihilator of any simple $M_0$-module $N$.
On the other hand, we have a unique integer $r_1$ such that $0 \leq \op{deg}(x_1/x_0^{r_1}) \leq d_0-1$.
If $\op{deg}(x_1/x_0^{r_1}) = 0 $, then $J = x_1/x_0^{r_1}I_{d_0}$ annihilates $N$.
Otherwise, the matrix
{\footnotesize
\[
J = \left( 
\begin{array}{c:c}  
 {\Huge O} & 
 {\begin{array}{ccc} x_1/x_0^{r_1} &  &\\ 
 & \ddots & \\
 & & x_1/x_0^{r_1} 
 \end{array}} \\ 
 \hdashline
  {\begin{array}{ccc} x_1/x_0^{r_1+1}& &  \\ 
  & \ddots& \\
   & &x_1/x_0^{r_1+1} \end{array}} & {\Huge O}
\end{array}
\right)
\in M_0
\]
}
annihilates $N$ because ${}^\exists n_J \in \mathbb{N}$ such that $J^{n_J} = (x_1^{h_1}/x_0^{h_0})I_{d_0}$  (the reduction of $N_i$ to $M_i$ is used here).
Thus, it is enough to prove that the global dimension of $(M_0/JM_0)_{\tilde{m}} = n-2$ (cf. the fourth paragraph of the proof of Lemma \ref{lem1-2}).
Note that we have
\vspace{-0.2mm}
{\small
\begin{align}
 M_0/JM_0 \simeq 
 \begin{pmatrix}
 F_{0,0} & F_{0,1} & \cdots & F_{0,d_0-1} \\
F_{0,-1} & F_{0,0} & \cdots & F_{0,d_0-2} \\
\vdots & \vdots & \cdots &\vdots \\
F_{0,-d_0+1} & F_{0,-d_0+2} & \cdots & F_{0,0} 
\end{pmatrix}, \label{isom}
\end{align}
}
where $F_0 = E_0/ x_1E_0 \simeq k\la x_0,x_2,\cdots,x_n \ra / (x_jx_i-q_{ji}x_ix_j, x_0^{h_0}+x_2^{h_2}+\cdots+x_n^{h_n})_{i,j}[x_0^{-1}]$ 
and $F_{0,j}$ is the degree $j$ part of $F_0$.

If any of $a_2, \cdots, a_n$ is not $0$, we can reduce to the case (a) from \ref{isom}.
If some of  $a_2, \cdots, a_n$ are $0$, repeat the above process until we can reduce to the case (a). 
\end{proof}

\begin{proof}[Proof of Theorem \ref{thm2}]
 $\op{gl.dim}(\op{qgr}(C/(f)))$ is finite.
So, the balanced dualizing complex $^{\phi}(C/(f))^1 [n]$ of $C/(f)$ induces the Serre functor of $\op{qgr}(C/(f))$ from {}\cite[Theorem A.4]{de2004ideal}.
We complete the proof as in the proof of Theorem \ref{thm1}.
\end{proof}

\section{Comparison and closed points}
In this section, we calculate closed points of noncommutative projective Calabi-Yau schemes of dimension $2$ obtained in Section \ref{122542_31Aug23} and compare our examples with commutative Calabi-Yau varieties and the first examples constructed in \cite{kanazawa2015}.
In particular, we show that a noncommutative projective Calabi-Yau scheme in Section \ref{122542_31Aug23} gives essentially a new example of noncommutative projective Calabi-Yau schemes.
\label{ptmods}






\begin{eg}
\label{wteg}
 Any weight $(d_0,d_1,d_2,d_3)$ of noncommutative projective Calabi-Yau 2 schemes in Theorem \ref{thm2} such that $\op{gcd}(d_0,d_1,d_2,d_3)=1$ is one of the following (obtained by using a computer):
\begin{align*}
 (d_0,d_1,d_2,&d_3) = (1,1,1,1),(1,1,1,3),(1,1,2,2),(1,1,2,4),(1,1,4,6), (1,2,2,5), \\ &(1,2,3,6), (1,2,6,9), (1,3,4,4), (1,6,14,21), (2,3,3,4), (2,3,10,15).
\end{align*}
\end{eg}
\vspace{-2.2mm}
From now, we focus on the closed points of noncommutatative projective Calabi-Yau 2 schemes in Theorem \ref{thm2} whose weights are of type $(1,1,a,b)$. 
We recall the notion of closed points of noncommutative projective schemes.

For simplicity, we often call an $\mathbb{N}$-graded $k$-algebra of the form $k \langle z_0, \cdots, z_m \rangle /(z_jz_i-p_{ji}z_iz_j)_{i,j} \ (p_{ji} \in k^\times, m \in \mathbb{N})$ with $\op{deg}(z_i)>0$
a  weighted quantum polynomial ring. 
$(p_{ji})$ is called the quantum parameter.

\begin{dfn}[{\cite[Section 3.1]{mori2015regular}}]
  Let $A$ be a finitely generated right noetherian connected $\mathbb{N}$-graded $k$-algebra.
A closed point of $\op{proj}(A)$ is an object of $\op{qgr}(A)$ represented by a $1$-critical module of $A$.
In particular, if $A$ is a quotient of a weighted quantum  polynomial ring, then every point is one of the following:
\begin{enumerate}
 \setlength{\itemsep}{0cm} 
 \item An ordinary point, which is represented by a finitely generated $1$-critical module of multiplicity $1$.
 \item A fat point, which is represented by a finitely generated $1$-critical module of multiplicity $>1$.
 \item A thin point, which is represented by a finitely generated $1$-critical module of multiplicity $<1$.
\end{enumerate}
For the definitions of $1$-critical modules and multiplicities, see \cite[Definition 3.1, 3.10]{mori2015regular}.
Note that if $A$ is generated in degree $1$, the notion of ordinary points and that of point modules are the same, and there is no thin point.
We denote by  $| \op{proj}(A) |$ the set of closed points of $\op{proj(A)}$.
\end{dfn}

Let $C := k \langle x_0,x_1,x_2,x_3 \rangle / (x_jx_i-q_{ji}x_ix_j)_{i,j}$ whose weight is of type $(d_0,d_1,d_2,d_3) = (1,1,a,b) \ (0 < a \leq b)$.
We assume that $q_{ij}q_{ji}=q_{ii}=1$ for all $i,j$.
Since  $d_0=1$, $C[x_0^{-1}]$ is strongly graded.
So, from \cite[Theorem 4.20]{mori2015regular}, we have
{
\[
 | \op{proj}(C) | = | \op{spec}(C[x_0^{-1}]_0) | \bigsqcup | \op{proj}(C/(x_0)) |, 
\]
}
where we denote by $| \op{spec}(C[x_0^{-1}]_0) |$ the set of simple modules of $C[x_0^{-1}]_0$.
In this equality, the $1$ (resp. $n>1$)-dimensional simple modules of $ \op{spec}(C[x_0^{-1}]_0)$ correspond to ordinary (resp. fat) points in $\op{proj}(C)$.
Similarly, we have
{
\begin{align*}
  | \op{proj}(C) | &= | \op{spec}(C[x_0^{-1}]_0) | \bigsqcup | \op{spec}(C/(x_0)[x_1^{-1}]_0)| \bigsqcup | \op{proj}(C/(x_0,x_1)) | .
\end{align*}
}
It easy to see that $C[x_0^{-1}]_0$ is isomorphic to $k \langle X_1,X_2,X_3 \rangle / (X_jX_i-q'_{ji}X_iX_j)_{i,j}$, where $q'_{ji} := q_{0j}^{d_i}q_{ji}q_{i0}^{d_j} \ (i,j \neq 0)$.
$C/(x_0)[x_1^{-1}]_0$ is also isomorphic to $k \langle Y_2,Y_3 \rangle / (Y_3Y_2-p_{32}Y_2Y_3)$, where $p_{32}:= q_{13}^{d_2}q_{32}q_{21}^{d_3}$.

Let $C_1:= k \langle x_0', x_1', x_2', x_3' \rangle /(x_j'x_i' - q'_{ji} x_i'x_j')_{i,j}$, where $\op{deg}(x_i')=1$, $q'_{0i}=q'_{j0}=1$ for all $i,j$.
Let $C_2:= k \langle y_1, y_2, y_3 \rangle /(y_jy_i -p_{ji} y_iy_j)_{i,j}$, where $\op{deg}(y_i')=1$, $p_{1i}=p_{j1}=1$ for all $i,j$.
Then, we can consider the point scheme of $\op{proj}(C_1)$ (resp. $\op{proj}(C_2)$), which is isomorphic to the set of ordinary points $| \op{proj}(C_1) |_{\text{ord}}$ (resp. $| \op{proj}(C_2) |_{\text{ord}}$) as sets.
Thus, we regard $| \op{proj}(C_1) |_{\text{ord}}$ (resp. $| \op{proj}(C_2) |_{\text{ord}}$) as the point scheme of $\op{proj}(C_1)$ (resp. $\op{proj}(C_2)$).

Let $| \op{spec}(C[x_0^{-1}]_0) |_1 $ (resp. $| \op{spec}(C/(x_0)[x_1^{-1}]_0) |_1$) be the set of $1$-dimensional simple modules of $C[x_0^{-1}]_0$ (resp. $C/(x_0)[x_1^{-1}]_0$).
Because $C_1[{x'_0}^{-1}]_0 \simeq C[x_0^{-1}]_0$ and $C_2[{y_1}^{-1}]_0 \simeq C/(x_0)[x_1^{-1}]_0$, 
we can think of $| \op{spec}(C[x_0^{-1}]_0) |_1 $ (resp. $|\op{spec}(C/(x_0)[x_1^{-1}]_0) |_1$) as a locally closed subscheme of $| \op{proj}(C_1) |_{\text{ord}}$ (resp. $| \op{proj}(C_2) |_{\text{ord}}$) from \cite[Theorem 4.20]{mori2015regular}.

\begin{lemm}
\begin{enumerate}
    \item If $q'_{ji} \neq 1$ for all $i,j \neq 0$, $| \op{spec}(C[x_0^{-1}]_0) |_1 $ is a union of three affine lines.  
    \item  If $p_{32} \neq 1$, $| \op{spec}(C/(x_0)[x_1^{-1}]_0) |_1$ is a union of two affine lines. 
    Otherwise, $| \op{spec}(C/(x_0)[x_1^{-1}]_0) |_1 \simeq \mathbb{A}^2$.
\end{enumerate}
\label{lem1}
\end{lemm}
\begin{proof}
 (2) is well-known (for example, see \cite[Section 4.3]{smith-noncommutative}).
Regarding (1),
under the assumption of the lemma, $\op{proj}(C_1)$ belongs to case (3) or case (4) in \cite[Corollary 5.1]{vitoria2010equivalences}.
This shows that $ |\op{spec}(C_1[{x'_0}^{-1}]_0) |_1$ is isomorphic to $\bigcup_{i \neq j} Z(X'_i,X'_j) \subset \mathbb{A}^3 = \op{Spec}(k[X'_1,X'_2,X'_3])$ (cf. \cite[Proposition 4.2]{vitoria2010equivalences} or \cite[Theorem 1]{belmans2016point}).
\end{proof}

\begin{rmk}
    We consider the weights $(1,1,a,b)$ and the quantum parameters which give noncommutative projective Calabi-Yau 2 schemes in Theorem \ref{thm2}. 
    Then, we can check that if $p_{32} \neq 1$, then $q'_{ji} \neq 1$ for all $i,j \neq 0$ by using a computer.
    Moreover, if $p_{32}=1$, then $q'_{ji} = 1$ for all $i,j \neq 0$.
    In this case, $| \op{spec}(C[x_0^{-1}]_0) |_1 \simeq \mathbb{A}^3$.
    \label{clptsncc}
\end{rmk}

We consider $C/(x_0,x_1) = k \langle x_2,x_3 \rangle /(x_3x_2-q_{32}x_2x_3)$. 
Then, it is known that a weighted quantum polynomial ring of $2$ variables is a twisted algebra of a commutative weighted polynomial ring $k[x,y]$ with $\op{deg}(x) = a >0,\op{deg}(y)=b>0$ (for example, see \cite[Example 4.1]{stephenson200070} or \cite[Example 3.6]{zhang1996twisted}).
So, it is enough to consider the 
 closed points of $\op{proj}(k[x,y])$.
We want to study the closed points of $\op{proj}(k[x,y])$ in the case of $(a,b) = (2,2),(2,4) \text{ or } (4,6)$. 
Note that when $(a,b) = (1,1)$ or $ (1,3)$, they are classified in \cite[Theorem 3.16]{mori2015regular}.
We treat a more general setting below.
\begin{lemm}
 Let $R=k[x,y]$ be a commutative weighted polynomial ring with $\op{deg}(x)=a>0,\op{deg}(y)=b>0$.
Let $g := \op{gcd}(a,b), a':= a/g$ and $b' := b/g$.
Then, every closed point of $\op{proj}(R)$ is one of the following:
\begin{enumerate}
 \setlength{\itemsep}{0cm} 
 \item $\pi R/(x)(-i), \;  i=0,\cdots,b-1$.
 \item $\pi R/(y)(-j), \; j=0, \cdots, a-1$.
 \item $\pi R/(\beta x^{b'}-\alpha y^{a'})(-k) ,$ where $(\alpha, \beta) \in \mathbb{P}^1 \backslash \{(0,1),(1,0)\}$ and $k= 0,\cdots, g-1$. 
\end{enumerate}
Moreover, all of them are not isomorphic in $\op{proj}(R)$.
\label{lem2}
\end{lemm}

\begin{proof}
 The proof is almost the same as the proof of \cite[Lemma 3.15, Theorem 3.16]{mori2015regular}.
We give the sketch of the proof.

Firstly, every closed point of $\op{proj}(R)$ is represented by a cyclic critical Cohen-Macaulay module of depth $1$.
Then, $M \in \op{gr}(R)$ satisfies these conditions and is generated in degree $0$ if and only if $M$ is isomorphic to one of $ R/(x), R/(y) \text{ or } R/(\beta x^{b'}-\alpha y^{a'} ) \; (\alpha, \beta \in k^\times)$.
Since being cyclic critical Cohen-Macaulay of depth $1$ is invariant under shifting, any closed point is represented by some shifts of one of the above modules (that is, $R/(x)(-l), R/(y)(-l), R/(\beta x^{b'}-\alpha y^{a'})(-l), \ l \in \mathbb{Z}$).

Finally, we classify the isomorphic classes of these modules in $\op{proj}(R)$.
We have no isomorphisms between the three types of closed points by considering their Hilbert polynomials and multiplicities.
Then, we have $\pi R/(\beta x^{b'}-\alpha y^{a'}) \simeq \pi R/(\beta x^{b'}-\alpha y^{a'})(-gl), \ ({}^\forall l \in \mathbb{Z}, {}^\forall (\alpha, \beta) \in \mathbb{P}^1 \backslash \{ (1,0), (0,1)\})$.
We also have $\pi R/(\beta x^{b'}-\alpha y^{a'}) \simeq \pi R/(\beta' x^{b'}-\alpha' y^{a'})$ if and only if $(\alpha, \beta) = (\alpha', \beta')$ in $\mathbb{P}^1 \backslash \{ (1,0), (0,1)\}$.
In addition, we can show that $\pi R/(x) \simeq \pi R/(x)(-i)$ (resp. $\pi R/(y) \simeq \pi R/(y)(-j)$) if and only if $i \equiv 0 \; (\op{mod} \; b) $ (resp. $j \equiv 0 \; (\op{mod} \; a)$).
From these discussions, we get the claim.
\end{proof}

We can study ordinary and thin points of noncommutative projective Calabi-Yau 2 schemes in Theorem \ref{thm2} by using the above investigations.
We give examples of noncommutative projective Calabi-Yau schemes whose moduli of ordinary closed points are different from those in \cite[Proposition 3.4]{kanazawa2015} and commutative Calabi-Yau varieties.

\begin{eg}
  We consider the weight $(1,1,2,2)$ and the quantum parameter 
\[
 \mathbf{q} = (q_{ij}) = 
 {\small
 \begin{pmatrix}
  1 & 1& 1 & \omega^2 \\
  1 & 1 & \omega^2 & 1 \\
  1 & \omega & 1 & 1 \\
 \omega & 1 & 1 & 1 \\
 \end{pmatrix}
 }
 , \quad \omega := \frac{-1+i\sqrt{3}}{2}.
\]
Then, we have 
\vspace{-3mm}
\begin{align*}
 \mathbf{q}' = (q_{ij}') = 
 {\small
\begin{pmatrix}
 1 & \omega^2 & \omega \\
 \omega  & 1 & \omega^2 \\
 \omega^2  & \omega & 1
\end{pmatrix}
}
, \quad
q_{13}^{d_2}q_{32}q_{21}^{d_3}= \omega^2.
\end{align*}
From Lemma \ref{lem1} and Lemma \ref{lem2}, the set of ordinary and thin points 
{
\begin{align*}
  | \op{proj}(C/(f)) |_{\text{ord \& thin}} &= | \op{spec}(C/(f)[x_0^{-1}]_0) |_1 \bigsqcup | \op{spec}(C/(f, x_0)[x_1^{-1}]_0)|_1 \\
  &\bigsqcup | \op{proj}(C/(f,x_0,x_1)) |
\end{align*}
}
is $24$ points. 
To be more precise, we have $| \op{spec}(C/(f)[x_0^{-1}]_0) |_1 = \bigsqcup_{i \neq j} Z(X_i,X_j, 1+X_1^6+X_2^3+X_3^3) \subset \mathbb{A}^3$, $| \op{spec}(C/(f,x_0)[x_1^{-1}]_0)|_1 = \bigsqcup_{i=1,2} Z(Y_i,1+Y_2^3+Y_3^3)$ and $| \op{proj}(C/(f,x_0,x_1)) | = \{3 \text{pts}\} \sqcup \{3 \text{pts}\}$. 

This calculation shows that for a fixed weight, if the number of the set of ordinary and thin points of $\op{proj}(C/(f))$ is finite, then the number is independent of the quantum parameters.
\label{eg2}
\end{eg}

From the method in Example \ref{eg2}, Remark \ref{clptsncc} and  a direct computation, we have the following.

\begin{prop}
 For a weight $(1,1,a,b)$ in Example \ref{wteg} and a quantum parameter $\mathbf{q}$ which gives a noncommutative projective Calabi-Yau scheme, if the set of ordinary and thin points of $\op{proj}(C/(f))$ is finite, then the number of the set is always $24$.
\end{prop}

The following proposition shows that some of noncommutative projective Calabi-Yau 2 schemes in Theorem \ref{thm2} are essentially new examples. 

\begin{prop}
\label{110549_8Sep23}
 There exists a noncommutative projective Calabi-Yau 2 scheme which is obtained in Theorem \ref{thm2} and not isomorphic to either commutative Calabi-Yau surfaces or noncommutative projective Calabi-Yau 2 schemes obtained in \cite{kanazawa2015}.
\end{prop}

\begin{proof}
We divide the proof into four steps.

\textit{Step 1.} We choose the weight $(1,1,a,b)$ and the quantum parameter $\mathbf{q}$ as in Example \ref{eg2}.
Then, the number of ordinary and thin points of $\op{proj}(C/(f))$ is finite.
So, $\op{proj}(C/f)$ is not isomorphic to any commutative Calabi-Yau surfaces.

\textit{Step 2.} We prove that $\op{proj}(C/(f))$ is not isomorphic to any noncommutative projective Calabi-Yau 2 schemes in \cite{kanazawa2015}.
To prove this, we use the theory established in \cite{burban2022morita}.
First, note that we can think of $\op{qgr}(C/(f))$ as the category of coherent modules of a sheaf $\mathcal{A}$ of algebras on the projective spectrum $\op{Proj}(k[s_0,s_1,s_2,s_3]/(s_0+s_1+s_2+s_3))$ (cf. the proof of Lemma \ref{lem4}). 
We define a sheaf $\mathcal{Z}_{\mathcal{A}}$ to be the sheaf whose sections are
{
\[
\Gamma(U,\mathcal{Z}_{\mathcal{A}}) = \{s \in \Gamma(U, \mathcal{A}) \mid s|_{V} \in Z(\Gamma(V,\mathcal{A})) , {}^\forall V \subset U : \text{open}\}
\] 
}
for all open subsets $U$ (cf. \cite[Proposition 2.11]{burban2022morita}).
In particular, if $U$ is affine, $\Gamma(U,\mathcal{Z}_{\mathcal{A}}) = Z(\Gamma(U,\mathcal{A}))$.
Then, we show that $\op{Spec}(Z(\Gamma(D_{+}(s_i), \mathcal{A})))$ has $4$ singular points when $i= 0,1$ and a $1$-dimensional singular locus when $i=2,3$.
In the following, we verify this claim for $i=0,2$. 
Similarly, the claim is proved for $i=1,3$.
In the following, we write $Z_i$ as $Z(\Gamma(D_{+}(s_i), \mathcal{A}))$ for any $i$.
We also use the notations in the proof of Lemma \ref{lem4}.

When $i=0$, any $m \in Z_0$ is of the form 
$
m = \left( \begin{smallmatrix}
    \mu_1 e & 0 \\
    0 & \mu_2 e
\end{smallmatrix} \right) \in N_0, \ (e \in E_{0,0}, \mu_1,\mu_2 \in k^\times )
$
from the definition of $\mathcal{A}$.
We have $E_{0,0} \simeq k \langle X_1,X_2,X_3 \rangle (X_jX_i-q'_{ji}X_iX_j, 1+X_1^6+X_2^3+X_3^3)_{i,j}$, which is obtained from the identifications $X_1 = x_1x_0^{-1},X_2=x_2x_0^{-2}$ and $X_3=x_3x_0^{-2}$. 
Here, the $q'_{ji}$ are as in Example \ref{eg2}.
So, $Z(E_{0,0}) \simeq k[Y,Z,W,U]/(1+Y^2+Z+W, YZW-\lambda_1 U^3) \ (\lambda_1 \in k^\times)$, which is obtained from the identifications $Y=(x_1x_0^{-1})^3, Z=(x_2x_0^{-2})^3, W=(x_3x_0^{-2})^3$ and $U=(x_1x_0^{-1})(x_2x_0^{-2})(x_3x_0^{-2})$.
On the other hand, we define the inclusion $\phi : Z(E_{0,0}) \rightarrow N_0$ in which $Y,Z,W$ are mapped naturally and $U$ to $
\left( \begin{smallmatrix}
    U & 0 \\
    0 & \omega U
\end{smallmatrix} \right) 
$.
It is easy to see that $\phi(Z(E_{0,0})) \subset Z_0 $.
Because the choice of $\mu_1$ determines $\mu_2$ in the above 
form of $m$, 
the map $\phi$ induces $Z_0\simeq Z(E_{0,0})$.
Thus, one can show that $\op{Spec}(Z_0)$ has $4$ singular points by using the Jacobi criterion.

When $i=2$, any $m \in Z_2$ is of the form 
$
m = \left( \begin{smallmatrix}
    \mu_1 e & 0 \\
    0 & \mu_2 e
\end{smallmatrix} \right) \in N_2, \ (e \in E_{2,0}, \mu_1,\mu_2 \in k^\times )
$
from the definition of $\mathcal{A}$. 
We also have $E_{2,0} \simeq k \langle X_0,X_1,X_2,X_3 \rangle/(X_jX_i-q''_{ji}X_iX_j, 1+X_0^6+X_1^6+X_3^3, X_0X_1-\lambda_2 X_2^2)_{i,j} \ (\lambda_2 \in k^\times)$, which is obtained from the identifications $X_0=x_0^2x_2^{-1}, X_1=x_1^2x_2^{-1}, X_2=x_0x_1x_2^{-1}$ and $X_3=x_3x_2^{-1}$. 
Here, the $q''_{ij}$ are defined by the matrix
\vspace{-2mm}
\[
(q''_{ij})=
\left( \begin{smallmatrix}
    1 & \omega & \omega^2 &\omega \\
    \omega^2 & 1 & \omega & \omega^2 \\
    \omega & \omega^2 & 1 & 1 \\
    \omega^2 & \omega  & 1 & 1
\end{smallmatrix}
\right).
\]
So, $Z(E_{2,0}) \simeq k[X,Y,W,U,V]/(X+Y+1+W, XY-\lambda_3 U^2,XYW- \lambda_4 V^2) \ (\lambda_3,\lambda_4 \in k^\times)$, which is obtained from the identifications $X=(x_0^2x_2^{-1})^3,Y=(x_1^2x_2^{-1})^3,W=(x_3x_2^{-1})^3, U=(x_0x_1x_2^{-1})^3$ and $ V=(x_0x_1x_2^{-1})(x_3x_2^{-1})$.
On the other hand, we define the inclusion $\phi: Z(E_{2,0})\rightarrow N_2$ in which $X,Y,W,U$ are mapped naturally and $V$ to $
 \left( \begin{smallmatrix}
    V & 0 \\
    0 & \omega V
\end{smallmatrix} \right)
$.
It is easy to see that $\phi(Z(E_{2,0})) \subset Z_2 $.
Because the choice of $\mu_1$ determines $\mu_2$ in the above form of $m$, 
the map $\phi$ induces $Z_2 \simeq Z(E_{2,0})$. 
Thus, one can show that $\op{Spec}(Z_2)$ has a $1$-dimensional singular locus by using the Jacobi criterion.

\textit{Step 3.} We consider the weight $(1,1,1,1)$ and take a quantum parameter which gives a noncommutative projective Calabi-Yau 2 scheme $\op{proj}(C'/(f'))$ whose point scheme is finite.
$\op{qgr}(C'/(f'))$ is thought of as the category of coherent modules of a sheaf $\mathcal{B}$ of algebras on the projective spectrum $\op{Proj}(k[t_0,t_1,t_2,t_3]/(t_0+t_1+t_2+t_3))$.

The number of the choices of quantum parameters $(q_{ij})$ which satisfy the conditions of Theorem \ref{thm2} and give a noncommutative projective Calabi-Yau scheme whose moduli space of point modules is finite is $20$ except permutating variables (we get the list below by using a computer and hand calculations):
\begin{align*}
&1.\left(\begin{smallmatrix}
 1 & 1 & 1 & 1 \\
 1 & 1 & -1 & -1 \\
1 & -1 & 1 & -1 \\
1 & -1 & -1 & 1 
\end{smallmatrix}\right),
2.\left(\begin{smallmatrix}
 1 & 1 &1 & 1 \\
1 & 1 & -i & i \\
1 & i & 1 & -i \\
1 & -i & i & 1 
\end{smallmatrix}\right),
3.\left(\begin{smallmatrix}
 1 & 1 & 1  & -1 \\
1 & 1 & -1 & 1 \\
1 & -1 & 1 & 1 \\
-1 & 1 & 1 & 1 
\end{smallmatrix}\right),
4.\left(\begin{smallmatrix}
 1 & 1 & 1  & -1 \\
1 & 1 & -i & -i \\
1 & i & 1 &i \\
-1 & i & -i & 1 
\end{smallmatrix}\right), \\
&5.\left(\begin{smallmatrix}
 1 & 1 & 1 & i \\
 1 & 1 & -1 & -i \\
 1 & -1 & 1 &-i \\
-i & i & i & 1 
\end{smallmatrix}\right),
6.\left(\begin{smallmatrix}
 1 & 1 & 1 & i \\
 1 & 1 & -i & -1 \\
1 & i& 1 & 1 \\
-i & -1 & 1 & 1 
\end{smallmatrix}\right),
7.\left(\begin{smallmatrix}
 1 & 1 & 1 & -i \\
 1 & 1 & -1 & i \\
 1  & -1 & 1 & i \\
i & -i & -i & 1 
\end{smallmatrix}\right),
8.\left(\begin{smallmatrix}
 1 & 1 & 1 & -i \\
 1 & 1 & -i & 1 \\
1 & i & 1 & -1 \\
i & 1 & -1 & 1 
\end{smallmatrix}\right), \\
&9.\left(\begin{smallmatrix}
 1 & 1 & -1 & -1 \\
 1 & 1 & -i & i \\
-1 & i & 1 & i \\
-1 & -i & -i & 1 
\end{smallmatrix}\right),
10.\left(\begin{smallmatrix}
 1 & 1 & -1 & i \\
1 & 1 & i & -1 \\
-1 & -i & 1 & -1 \\
-i & -1 & -1 & 1 
\end{smallmatrix}\right),
11.\left(\begin{smallmatrix}
 1 & 1 & -1 & -i \\
1 & 1 & -i & -1 \\
-1 & i & 1 & -1 \\
i & -1 & -1 & 1 
\end{smallmatrix}\right),
12.\left(\begin{smallmatrix}
 1 & 1 & i & i \\
1 & 1 & -i & -i \\
-i & i & 1 & -1 \\
-i & i & -1 & 1 
\end{smallmatrix}\right), \\
&13. \left(\begin{smallmatrix}
 1 & 1 & i & -i \\
1 & 1 & -i & i \\
-i & i & 1 & 1 \\
i & -i & 1 & 1 
\end{smallmatrix}\right),
14. \left(\begin{smallmatrix}
 1 & -1 & -1 & -1 \\
-1 & 1 & -1 & -1 \\
-1 & -1 & 1 & -1 \\
-1 & -1 & -1 & 1 
\end{smallmatrix}\right),
15. \left(\begin{smallmatrix}
 1 & -1 & -1 & -1 \\
-1 & 1 & -i & i \\
-1 & i & 1 & -i \\
-1 & -i & i & 1 
\end{smallmatrix}\right),
16. \left(\begin{smallmatrix}
 1 & -1 & -1 & i \\
-1 & 1 & -1 & i \\
-1 & -1 & 1 & i \\
-i & -i & -i & 1 
\end{smallmatrix}\right), \\
&17. \left(\begin{smallmatrix}
 1 & -1 & -1 & -i \\
-1 & 1 & -1 & -i \\
-1 & -1 & 1 & -i \\
i & i & i & 1 
\end{smallmatrix}\right),
18. \left(\begin{smallmatrix}
 1 & -1 & i & i \\
-1 & 1 & i & i \\
-i & -i & 1 & -1 \\
-i & -i & -1 & 1 
\end{smallmatrix}\right),
19.\left(\begin{smallmatrix}
 1 & i & i & i \\
-i & 1 & -i & i \\
-i & i & 1 & -i \\
-i & -i & i & 1 
\end{smallmatrix}\right),
20. \left(\begin{smallmatrix}
 1 & i & i & -i \\
-i & 1 & -i & -i \\
-i & i & 1 & i \\
i & i & -i & 1 
\end{smallmatrix}\right). 
\end{align*}

When we choose one $(q_{ij})$ of the above 20 quantum parameters, then for any $l$, $ \Gamma(D_{+}(t_l),\mathcal{B}) \simeq k \langle Y_1,Y_2,Y_3 \rangle/ (Y_iY_j-q_{ij}'Y_jY_i, Y_1^4+Y_2^4+Y_3^4+1)_{1 \leq i,j \leq 3}$, where $(q'_{ij})$ is represented by one of the following matrices (we can verify this with direct calculations) :
\[
(a).  \left(\begin{smallmatrix}
 1 & -1 & -1  \\
-1 & 1 & -1  \\
-1 & -1 & 1 
\end{smallmatrix}\right), \qquad
(b). \left(\begin{smallmatrix}
 1 & -i & i  \\
i & 1 & -i  \\
 -i & i & 1 
\end{smallmatrix}\right).
\]

We write  $Z_l':=Z(\Gamma(D_{+}(t_l),\mathcal{B})) $.
When $(q'_{ij})$ is type (a), $\op{Spec}(Z'_l)$
has 6 singular points because $Z_l'$ is generated by $Y_1^2,Y_2^2,Y_3^2$ and $Y_1Y_2Y_3$ as a $k$-algebra.
When $(q'_{ij})$ is type (b), $\op{Spec}(Z_l')$ has 3 singular points because $Z_l'$ is generated by $Y_1^4,Y_2^4,Y_3^4 $ and $Y_1Y_2Y_3$ as a $k$-algebra. 
Moreover, for any $(q_{ij})$ in the above table, if $\mathcal{B}$ is type (a) (resp. (b)) on $D_+(t_l)$ for some $l$, it is also type (a) (resp. (b)) on $D_+(t_l)$ for any other $l$.

\textit{Step 4.} If $\op{qgr}(C/(f))$ is equivalent to $\op{qgr}(C'/(f'))$ then, 
 we must have an isomorphism of schemes between $\op{Spec}(\mathcal{Z}_{\mathcal{A}})$ and $\op{Spec}(\mathcal{Z}_{\mathcal{B}})$ by \cite[Theorem 4.4]{burban2022morita} (cf.  \cite[Section 6]{artin1994}).
Since $\op{Spec}(\mathcal{Z}_{\mathcal{A}})$ has infinite singular points, but, $\op{Spec}(\mathcal{Z}_{\mathcal{B}})$ has finite singular points, such a situation does not happen.
Hence, we complete the proof.
\end{proof}


\subsubsection*{Acknowledgements}
The author would like to express his gratitude to his supervisor Professor Hajime Kaji for his encouragement.
He is also grateful to Professor Atsushi Kanazawa for telling him the articles \cite{liu2019donaldson}, \cite{liu2020donaldson}.
He would like to thank Professor Izuru Mori, Professor Bal\'{a}zs Szendroi, Professor Ryo Ohkawa, Professor Shinnosuke Okawa and Professor Kenta Ueyama for helpful comments.
In addition, he thanks Niklas Lemcke for proofreading his English.
This work is supported by Grant-in-Aid for JSPS Fellows (Grant Number 22KJ2923).

\printbibliography 


\end{document}